\documentclass[12pt,reqno]{amsart}

\usepackage{epsfig}
\usepackage{amsmath,amsthm}
\usepackage{amsfonts}
\usepackage{latexsym}
\usepackage{amsbsy}
\usepackage{amsfonts,amssymb,amscd,amsmath,enumerate,verbatim,calc,eucal,enumerate}

\numberwithin{equation}{section}

\usepackage{amsfonts,amssymb,amscd,amsmath,enumerate,verbatim,calc,eucal,enumerate}
\newcommand{\ma}{\mathcal}
\newcommand{\mf}{\mathfrak}
\newcommand{\m}{\CMcal}

\theoremstyle{plain}
\newtheorem{theorem}{Theorem}[section]
\newtheorem{corollary}[theorem]{Corollary}
\newtheorem{lemma}[theorem]{Lemma}
\newtheorem{proposition}[theorem]{Proposition}

\newtheorem{definition}[theorem]{Definition}

\begin{document}
\title[Bures distance and transition probability for $\alpha$-CPD-kernels]{Bures distance and transition probability \\for $\alpha$-CPD-kernels}
\date{\today}

\author{Santanu Dey}
\address{Department of Mathematics, Indian Institute of Technology Bombay, Powai, Mumbai-400076, India} 
\email{santanudey@iitb.ac.in}

\author{Harsh Trivedi}
\address{Statistics and Mathematics Unit, Indian Statistical Institute, 8th Mile, Mysore Road, Bangalore, 560059, India}
\email{trivediharsh26@gmail.com}

\begin{abstract}
If the symmetry (fixed invertible self adjoint map) of Krein spaces is replaced by a fixed unitary, then we obtain the notion of S-spaces which was
introduced by Szafraniec. 
Assume $\alpha$ to be 
an automorphism on a $C^*$-algebra. In this article, we obtain the Kolmogorov decomposition of $\alpha$-completely positive definite (or $\alpha$-CPD-kernels for short) and investigate
the Bures distance between $\alpha$-CPD-kernels. 
We also define transition probability for these kernels and find a characterization of the transition probability.
\end{abstract}


\maketitle

\section{Introduction}

The Gelfand-Naimark-Segal (GNS) construction 
for a state on a $C^*$-algebra yields us a representation of the $C^*$-algebra 
on a Hilbert space and a cyclic vector. Bures \cite{B69} gave a distance formula, between given two states on a $C^*$-algebra, 
which is equal to the infimum of norms of differences between the cyclic vectors of corresponding GNS constructions where the infimum is taken over all GNS constructions with common representation space.
A linear map $\tau$ from a $C^*$-algebra $\m B$ to a $C^*$-algebra 
$\m C$ is said to be {\it completely positive} if $\sum_{i,j=1}^n c_j^{*}
\tau(b_j^{*}b_i)c_i\geq 0$ whenever $b_1,b_2,\ldots,b_n\in\m B$; $c_1,c_2,\ldots,c_n\in\m C$ and $n\in \mathbb N$. Stinespring's 
theorem (cf. \cite[Theorem 1] {St55}) 
and Paschke's GNS construction (cf. \cite[Theorem 5.2] {Pas73}) characterize operator valued completely positive maps and those
completely positive maps which takes values in $C^*$-algebras, respectively in a similar way as the GNS construction characterizes states. If we choose 
the completely 
positive maps to be states, then these two constructions coincide with the GNS construction. 
Motivated by the formulation of Bures, a distance
formula is defined in \cite{KSWR08} between $\ma B(\m H)$-valued completely positive maps where the formula is in terms of Stinespring's constructions. 
Distance formulas between two completely positive
maps are useful and have applications in operator theory, quantum information science and mathematical physics 
(cf. \cite{BSu13}, \cite{JKP09} and \cite{H15}, respectively).

Kaplan defined multi-states and proved the GNS construction for them in \cite{Kap89}. In a similar way,
Heo \cite{He99} defined completely multi-positive maps which extends the terminology of completely positive maps. We recall the definition of completely positive definite kernels
(cf. \cite{BBLS04}) which generalizes the notion of completely multi-positive maps (cf. \cite[Note 4.7] {Sk11}):
Let us denote the set of all bounded linear maps from a $C^*$-algebra $\m B$ to a $C^*$-algebra $\m C$ by $\ma
B(\m B,\m C)$. For a set $\Omega$
we say that a mapping $\mf{K}:\Omega\times \Omega \to \ma B(\m B,\m C)$ is a {\it completely positive definite kernel} or a {\it CPD-kernel} over $\Omega$ from
$\m B$ to $\m C$ if\\
\[
\sum_{i,j} c^*_i \mf{K}^{\sigma_i, \sigma_j} (b^*_i b_j) c_j \geq 0~\mbox{for all finite choices of $\sigma_i\in \Omega$, $b_i\in \m B$, $c_i\in \m C$. }~
\]
The dilation theory of these kernels were explored extensively in \cite{BBLS04,DH14}.

Assume $(E,\langle\cdot,\cdot\rangle)$ 
to be a Hilbert $\m{A}$-module where $\m{A}$ is a $C^{\ast}$-algebra and $J$ to be an invertible adjointable map on $E$ such that $J=J^*=J^{-1}$. 
Define a map $[\cdot,\cdot]:E\times E\to \m A$ by
\begin{eqnarray}
[x,y]: =\langle Jx,y\rangle~\mbox{for all}~x,y\in E.
\end{eqnarray}
The triple $(E,\m A,J)$ is called a {\it Krein $\m{A}$-module} which extends the notion of Krein spaces and $J$ is called the {\it symmetry}. In \cite{HJ10} 
a dilation theorem for $\alpha$-completely positive maps, where $\alpha$ is an automorphism, was obtained 
in terms of representations on Krein $C^*$-modules. 
The notion of S-modules 
(cf. \cite{DT15}) which is defined below,
extends the notion of S-spaces \cite{Sz09} and Krein $\m A$-modules:
\begin{definition}
	Let $(E,\langle\cdot,\cdot\rangle)$ be a Hilbert $\m{A}$-module where $\m{A}$ is a $C^{\ast}$-algebra and let $U$ be a unitary on $E$,
	i.e., $U$ is an invertible adjointable map on $E$ such that $U^*=U^{-1}$. Then we can define an
	$\m{A}$-valued sesquilinear form by
	\begin{eqnarray}
	[x,y]: =\langle x,Uy\rangle~\mbox{for all}~x,y\in E.
	\end{eqnarray}
	In this case we call {\rm $(E,\m A, U)$ is an S-module}.
\end{definition}

S-correspondence is an analogue of $C^*$-correspondence in the context of S-modules. In our earlier work we introduced the notion of $\alpha$-completely positive 
definite or $\alpha$-CPD-kernels and for any $\alpha$-CPD-kernel $\mathfrak K$ obtained a partial decomposition theorem using 
reproducing kernel S-correspondences.
Bhat and Sumesh explored 
Bures distance 
between more general completely
positive maps in \cite{BSu13} and their approach was based on von Neumann modules. In this article we
study the Bures distance 
formula between two $\alpha$-CPD-kernels over a set $\Omega$ from $\m B$ to $\m C$ when $\m C$ is a von Neumann algebra. For this we first obtain an important decomposition, called the Kolmogorov decomposition, of $\alpha$-CPD-kernels over $\Omega$ in terms of a tuple consisting of an S-correspondence ${\mathcal F}$ and a map from $\Omega$ to ${\mathcal F}$.

It is shown in section 4 that certain intertwiners between (minimal) Kolmogrorov decompositions of two $\alpha$-CPD-kernels can be used to compute the Bures distance between the $\alpha$-CPD-kernels.  Suppose $\m B$ is a von Neumann algebra. The rigidity theorem in this section establishes that if the Bures distance between a CPD-kernel over $\Omega$ from $\m B$ to $\m B$ and the identity kernel is bounded by certain constant, then the Kolmogorov decomposition of the CPD-kernel contains  a copy of $\m B$.
The notion of transition probability between two states over a unital $*$-algebra is introduced by Uhlmann in \cite{Uh76}. 
Alberti \cite{Al83} found 
several techniques to compute the transition probability between two states over a unital $C^*$-algebra. 
Recently Heo \cite{H15} developed a Bures distance formula between $\alpha$-CP maps and considered unbounded representations of a $*$-algebra on 
a Krein space
to study transition probability between P-functionals.
In the last section we define the notion of transition probability between two $\alpha$-CPD-kernels and do an analysis of this notion
based on results known for transition probability between two states, etc.

\subsection{Background and notations}

We recall the definitions of Hilbert $C^*$-modules and von Neumann modules: 

\begin{definition}
	Assume $\m B$ to be a $C^*$-algebra. Let $E$ be a complex vector space which is a
	right $\m B$-module such that the module action is compatible with the scalar product. The module $E$ is called 
	a {\rm Hilbert $C^*$-module over $\m B$} or a  {\rm Hilbert $\m B$-module} if there exists a mapping $\langle
	\cdot,\cdot \rangle : E \times E \to \m{B}$ satisfying the following conditions:
	\begin{itemize}
		\item [(i)] $\langle x,x \rangle \geq 0 ~\mbox{for}~ x \in {E}  $ and $\langle x,x \rangle = 0$ only if $x = 0 ,$
		\item [(ii)] $\langle x,yb \rangle = \langle x,y \rangle b ~\mbox{for}~ x,y \in {E}$ and $~\mbox{for}~ b\in \m B,  $
		\item [(iii)]$\langle x,y \rangle=\langle y,x \rangle ^*~\mbox{for}~ x ,y\in {E} ,$
		\item [(iv)]$\langle x,\mu y+\nu z \rangle = \mu \langle x,y \rangle +\nu \langle x,z \rangle ~\mbox{for}~ x,y,z \in {E} $ and for 
		$\mu,\nu \in \mathbb{C},$
		\item [(v)] $E$ is complete with respect to the norm $\| x\| :=\|\langle
		x,x\rangle\|^{1/2} ~\mbox{for}~ x \in {E}$. 
	\end{itemize} 
\end{definition}

Let $\ma B(\m H,\m H')$ be the space of all bounded linear operators from Hilbert space $\m H$ to Hilbert space $\m H'$. 
If $\m B$ is a von Neumann algebra acting on a Hilbert space $\m H,$ then for each Hilbert $\m
B$-module $E$ the interior tensor product $E\bigotimes \m H$ is a Hilbert space. Fix $x\in E$ and define a bounded linear
map $L_x$ from $\m H$ to $E\bigotimes \m H$ by
$$L_x (h):=x\otimes h~\mbox{for}~ h\in \m H.$$
We identify each $x\in E$ with $L_x$, because $L^*_{x_1} L_{x_2} =\langle x_1,x_2\rangle~\mbox{for all}~
x_1,x_2\in E$. Therefore  
$E$ is a concrete submodule of $\ma B(\m H,E\bigotimes \m H)$. 
\begin{definition}
	The Hilbert $\m
	B$-module $E$ is called a {\rm von Neumann $\m B$-module}
	or a {\rm von Neumann
		module over $\m B$} if it is strongly closed in $\ma B(\m
	H,E\bigotimes \m H)$. 
\end{definition}
Every von Neumann
$\m B$-module is complemented in Hilbert B-modules
which contains it as a $\m B$-submodule because all von Neumann $\m B$-modules are self-dual. \cite{Sk06} contains a detailed exposition on von Neumann modules. In particular, if we denote the set of all {\it adjointable
	maps} on $E$ by $\ma B^a(E)$, then the map from $\ma B^a(E)$ to $\ma B(E\bigotimes \m H)$ defined by $a\mapsto a\otimes id_{\m H}$ is a unital 
$*$-homomorphism, and hence it is an isometry. 
Therefore $\ma B^a(E)\subset \ma B(E\bigotimes \m H)$.

We further recall the definition
of an $\alpha$-completely positive definite kernel, which is central to our study here, from \cite{DT15}:

\begin{definition}\label{def1}
	Suppose $\m B$ and $\m C$ are unital $C^*$-algebras. We denote the set of all bounded linear maps from  $\m B$ to $\m C$ by $\ma
	B(\m B,\m C)$. Let
	$\alpha$ be an automorphism on $\m B$, i.e., $\alpha:\m B\to \m B$ is a bijective unital $*$-homomorphism. For a set $\Omega$, by a {\rm kernel} 
	$\mf K$ over $\Omega$ from $\m B$ to $\m C$ we mean a function $\mf{K}:\Omega\times \Omega \to \ma B(\m B,\m C)$, 
	and $\mf K$ is called {\rm Hermitian} if $\mf{K}^{\sigma, \sigma'}(b^*)=\mf{K}^{\sigma', \sigma}(b)^*$ for all $\sigma,\sigma'\in \Omega$ 
	and $b\in\m B$.
	We say that a Hermitian kernel $\mf K$ over $\Omega$ from $\m B$ to $\m C$ is an {\rm $\alpha$-completely positive definite kernel} or
	an {\rm $\alpha$-CPD-kernel} over $\Omega$ from $\m B$ to $\m C$ if for finite choices $\sigma_i\in \Omega$, $b_i\in \m B$, $c_i\in \m C$ we have
	\begin{itemize}
		\item [(i)] $ \sum_{i,j} c^*_i \mf{K}^{\sigma_i, \sigma_j} (\alpha(b_i)^* b_j) c_j \geq 0, $
		\item [(ii)] $ \mf{K}^{\sigma_i, \sigma_j} (\alpha(b))= 
		\mf{K}^{\sigma_i, \sigma_j}(b) $ for all $b\in\m B$,
		\item [(iii)] for each $b\in\m B$ there exists $M(b)>0$ such that 
		\begin{eqnarray*}
			\left\|\sum\limits_{i,j=1}^n c^*_i  \mf{K}^{\sigma_i,\sigma_j}(\alpha(b^*_ib^*)bb_j) c_j\right\|
			&\leq & M(b)\left\|\sum\limits_{i,j=1}^n   c^*_i \mf{K}^{\sigma_i,\sigma_j}(\alpha(b^*_i)b_j) c_j\right\|.
		\end{eqnarray*}
	\end{itemize}
	We use the notation $\m K^{\alpha}_{\Omega}(\m B,\m C)$ for the set of all $\alpha$-CPD-kernels over $\Omega$ from $\m B$ to $\m C$. 
\end{definition}

\section{Kolmogorov decomposition for $\alpha$-CPD-kernels}\label{sec1}

Assume $E_1$ and $E_2$ are Hilbert $C^*$-modules over a $C^*$-algebra $\m B$ such that $\left( {E}_{1},\m B,U_{1}\right) $ and $\left( E
_{2},\m B,U_{2}\right) $ are S-modules. Then for each adjointable operator $T$ from
${E}_{1}$ to ${E}_{2}$, there exists an operator $T^{\natural}$ from $
{E}_{2}$ to ${E}_{1}$ satisfying the following
\[
\langle T(x),U_2 y\rangle=\langle x,U_1T^\natural(y)\rangle~\mbox{for all}~x\in E_1,~y\in E_2.
\]
For instance, $T^{\natural}=U^*_{1}T^{\ast }U_{2}$. 
\begin{definition} Let $\m A$ and $\m{B}$ be unital $C^{\ast}$-algebras, and let $(E,\m B,U)$ be an S-module. 
\begin{itemize}
 \item [(i)] An algebra homomorphism $\pi :\m{A}\rightarrow
\mathcal{B}^a({E})$ is called an {\rm $U$-representation of $\m{A}$ on $(E,\m B,U)$} if
$\pi (a^{\ast })=U^*\pi (a)^{\ast }U=\pi (a)^{\natural}$, i.e., $$[\pi (a)x ,y]=[x ,\pi (a^{\ast })y]~\mbox{for all}~x,y\in E.$$
\item [(ii)] The S-module $(E,\m B,U)$ is called an {\rm S-correspondence from $\m A$ 
 to $\m B$} if there exists a $U$-representation $\pi$ of $\m{A}$ on $(E,\m B,U)$, i.e., $E$ is also a left $\m A$-module with 
 \[
  ax:=\pi(a)x~\mbox{for all}~a\in\m A,x\in E.
 \]
\end{itemize}

 \end{definition}
 In the next theorem we obtain the Kolmogorov decomposition for an $\alpha$-CPD-kernel. 
  \begin{theorem}\label{thm1}
Assume $\m C$ to be a $C^*$-algebra and $\Omega$ to be a set. Suppose $\alpha$ is an automorphism on a unital $C^*$-algebra $\m B$
and $\mf{K}:\Omega\times \Omega\to \ma B(\m B,\m C)$ is a Hermitian kernel. 
Then the following conditions are equivalent:
  \begin{itemize}
 \item [(i)] The kernel $\mf{K}$ is an $\alpha$-CPD-kernel.
 \item [(ii)] There exists a pair $(\m F, \mf{i})$ consisting of an S-correspondence $\m F$ from $\m B$ to $\m C$ and 
a map $\mf{i}: \Omega\to \m F$ such that $ \overline{\mbox{span}}\{b\mf{i}(\sigma) c:b\in\m
B,~\sigma\in\Omega,~c\in\m C\}=\m F$ and
\begin{align}\label{eqn*}  \mf{K}^{\sigma,\sigma'}(b)= \langle \mf{i}(\sigma), b
\mf{i}(\sigma')\rangle=\langle \alpha(b^*)\mf{i}(\sigma), 
\mf{i}(\sigma')\rangle~\mbox{for}~\sigma, \sigma'\in \Omega;~b\in \m B.
 \end{align}
\end{itemize}
\end{theorem}
\begin{proof}

We assume the statement (ii) holds.  Then from Equation \ref{eqn*} it follows that 
\begin{align*}
  \displaystyle\sum_{i=1}^n \displaystyle\sum_{j=1}^n c^*_i \mf{K}^{\sigma_i, \sigma_j} (\alpha(b^*_i) b_j) c_j
 &=\displaystyle\sum_{i=1}^n \displaystyle\sum_{j=1}^n c^*_i  \langle \mf{i}(\sigma_i), \alpha(b^*_i) b_j
\mf{i}(\sigma_j)\rangle c_j
=\displaystyle\sum_{i=1}^n \displaystyle\sum_{j=1}^n c^*_i  \langle b_i\mf{i}(\sigma_i),  b_j
\mf{i}(\sigma_j)\rangle c_j 
\\ &= \left<\displaystyle\sum_{i=1}^n b_i\mf i(\sigma_i)c_i
, \displaystyle\sum_{j=1}^n b_j \mf i(\sigma_j)c_j\right>
 \geq 0
\end{align*}
for all $\sigma_1,\sigma_1,\ldots,\sigma_n\in \Omega$, $b_1,b_2,\ldots,b_n\in \m B$, $c_1,c_2,\ldots,c_n\in \m C$. Further, 
for all $b\in\m B$ and $\sigma,\sigma'\in \Omega$ we get
\begin{align*}
 \mf{K}^{\sigma, \sigma'}(\alpha(b))
 &=\langle \mf{i}(\sigma), \alpha(b)
\mf{i}(\sigma')\rangle
 =\langle b^*\mf{i}(\sigma), 
\mf{i}(\sigma')\rangle
\\ &=(\langle \mf{i}(\sigma'),b^* \mf{i}(\sigma)\rangle)^*=\mf{K}^{\sigma', \sigma}(b^*)^*=\mf{K}^{\sigma, \sigma'}(b).
\end{align*}
Finally, for a fixed $b\in \m B$ and each $\sigma_1,\sigma_1,\ldots,\sigma_n\in \Omega$, $b_1,b_2,\ldots,b_n\in \m B$ and $c_1,c_2,\ldots,c_n\in \m C$ we obtain
\begin{align*}
 &\left\|\sum\limits_{i,j=1}^n c^*_i  \mf{K}^{\sigma_i,\sigma_j}(\alpha(b^*_ib^*)bb_j) c_j\right\|
 =\left\|\sum\limits_{i,j=1}^n c^*_i  \langle \mf{i}(\sigma_i), \alpha(b^*_ib^*)bb_j
\mf{i}(\sigma_j)\rangle c_j
\right\|
  \\&=\left\|\sum\limits_{i,j=1}^n c^*_i  \langle b_i\mf{i}(\sigma_i), \alpha(b^*)bb_j
\mf{i}(\sigma_j)\rangle c_j
\right\|
=\left\| \left< \sum_{i=1}^n b_i\mf{i}(\sigma_i)c_i, \alpha(b^*)b\left(\sum_{j=1}^n b_j
\mf{i}(\sigma_j) c_j\right)\right>
\right\|
     \\&\leq\|\alpha(b)^*b\|\left\| \sum_{i=1}^n b_i\mf{i}(\sigma_i)c_i \right\|^2
       \leq\|b\|^2\left\|\left< \sum_{i=1}^n b_i\mf{i}(\sigma_i)c_i,\sum_{j=1}^n b_j
\mf{i}(\sigma_j) c_j\right> \right\|
       \\&=\|b\|^2\left\|\sum\limits_{i,j=1}^n   c^*_i \mf{K}^{\sigma_i,\sigma_j}(\alpha(b^*_i)b_j) c_j\right\|.
\end{align*}
Thus the function $\mf{K}$ is an $\alpha$-CPD-kernel, i.e., (i) holds.

Conversely, assume that the statement (i) holds. Let $\Omega_{\mathbb{C}}$ be the vector space $\bigoplus_{\sigma\in \Omega}\mathbb C$, i.e., 
 $$\Omega_{\mathbb{C}}=\{(\lambda_{\sigma})_{\sigma\in \Omega}: \lambda_{\sigma}~\mbox{is non-zero for finitely many}~\sigma\in \Omega\}.$$
 We denote the element $(\delta_{\sigma,\sigma'})_{\sigma'\in \Omega}$ of $\Omega_{\mathbb{C}}$ by $e_{\sigma}$ for each $\sigma\in \Omega$. 
 The vector space tensor product $\m F_0:=\m B\bigotimes \Omega_{\mathbb{C}}\bigotimes\m C$ is a $\m B$-$\m C$ bimodule in a natural way. Define a sesquilinear mapping
 $\langle\cdot,\cdot\rangle:\m F_0\times \m F_0\to \m C$ by
 $$\left< \displaystyle\sum_{i=1}^nb_i\otimes e_{\sigma_i} \otimes c_i,\displaystyle\sum_{j=1}^m b^{\prime}_j\otimes e_{\sigma^{\prime}_j} \otimes c^{\prime}_j\right>:=
 \displaystyle\sum_{i=1}^n \displaystyle\sum_{j=1}^m c^{*}_i\mf K^{\sigma_i,\sigma^{\prime}_j}(\alpha(b_i)^*b^{\prime}_j)c^{\prime}_j$$
for all $b_i,b^{\prime}_j\in\m B;~c_i,c^{\prime}_j\in\m C;\sigma_i,\sigma^{\prime}_j\in \Omega$ where $1\leq i\leq n$ and $1\leq j\leq m$. The map $\langle\cdot,\cdot\rangle$ is in fact positive definite, 
since the kernel $\mf K$ is an $\alpha$-CPD-kernel (cf. Definition \ref{def1} $(i)$). 
Using the Cauchy-Schwarz inequality for positive-definite sesquilinear forms we conclude that 
$$K :=\left\{ \displaystyle\sum_{i=1}^nb_i\otimes e_{\sigma_i} \otimes c_i\in \m F_0:
\displaystyle\sum_{i=1}^n \displaystyle\sum_{j=1}^n c^{*}_i\mf K^{\sigma_i,\sigma_j}(\alpha(b_i)^*b_j)c_j =0\right\}$$
is a submodule of
$\m F_0$. Therefore $\langle\cdot,\cdot\rangle$ induces
canonically on the quotient module $\m F_0/ K$ a $\m C$-valued inner product. 
Henceforth we denote this induced inner-product by $\langle\cdot,\cdot\rangle$ itself. Let $\m F$ be
the Hilbert $\m{C}$-module obtained by the
completion of $\m F_0/K$. 

Define a linear map $U:\m F\to\m F$ by
\[
U\left( \displaystyle\sum_{i=1}^nb_i\otimes e_{\sigma_i} \otimes c_i+K\right) \
=\displaystyle\sum_{i=1}^n \alpha(b_i)\otimes e_{\sigma_i} \otimes c_i +K~\mbox{where}~b_i\in \m B,~c_i\in \m C~\mbox{and}~\sigma_i\in \Omega.  
\]
Here $U$ is a unitary, since
\begin{eqnarray*}
 &&\left< U\left( \displaystyle\sum_{i=1}^nb_i\otimes e_{\sigma_i} \otimes c_i+K\right) ,
 U\left( \displaystyle\sum_{j=1}^m b^{\prime}_j\otimes e_{\sigma^{\prime}_j} \otimes c^{\prime}_j+K\right) \right>
 \\ &=&\left<\displaystyle\sum_{i=1}^n \alpha(b_i)\otimes e_{\sigma_i} \otimes c_i+K , \displaystyle\sum_{j=1}^m \alpha(b^{\prime}_j)\otimes e_{\sigma^{\prime}_j} \otimes c^{\prime}_j+K \right>
 \\&=&\displaystyle\sum_{i=1}^n \displaystyle\sum_{j=1}^m c^{*}_i\mf K^{\sigma_i,\sigma^{\prime}_j}(\alpha(\alpha(b_i))^*\alpha(b^{\prime}_j))c^{\prime}_j 
 =\displaystyle\sum_{i=1}^n \displaystyle\sum_{j=1}^m c^{*}_i\mf K^{\sigma_i,\sigma^{\prime}_j}(\alpha(b_i)^*b^{\prime}_j)c^{\prime}_j 
\\&=&\left< \displaystyle\sum_{i=1}^nb_i\otimes e_{\sigma_i} \otimes c_i+K , \displaystyle\sum_{j=1}^mb^{\prime}_j\otimes e_{\sigma^{\prime}_j} \otimes c^{\prime}_j+K\right>,
\end{eqnarray*}
for all $b_i,b^{\prime}_j \in\m B,~c_i,c^{\prime}_j\in\m C,~ \sigma_i,\sigma^{\prime}_j\in\Omega$ for $1\leq i\leq n,~1\leq j\leq m$, and since 
$U$ is surjective. In a similar way it follows that the linear map $\displaystyle\sum_{j=1}^m b_i\otimes e_{\sigma^{\prime}_j}
\otimes c^{\prime}_j
+K\mapsto\displaystyle\sum_{j=1}^m\alpha^{-1}(b^{\prime}_j)\otimes e_{\sigma^{\prime}_j} \otimes c^{\prime}_j+K$ is isometric and hence well-defined.
Since
\begin{eqnarray*}
 &&\left< U\left( \displaystyle\sum_{i=1}^nb_i\otimes e_{\sigma_i} \otimes c_i+K\right) ,\displaystyle\sum_{j=1}^m b^{\prime}_j\otimes e_{\sigma^{\prime}_j} \otimes c^{\prime}_j+K \right>
 \\&=&\left< \displaystyle\sum_{i=1}^n\alpha(b_i)\otimes e_{\sigma_i} \otimes c_i+K,\displaystyle\sum_{j=1}^m b_i\otimes e_{\sigma^{\prime}_j} \otimes c^{\prime}_j+K \right>
\\&=&\displaystyle\sum_{i=1}^n \displaystyle\sum_{j=1}^m c^{*}_i\mf K^{\sigma_i,\sigma^{\prime}_j}(\alpha(\alpha(b_i))^*b^{\prime}_j)c^{\prime}_j
=\displaystyle\sum_{i=1}^n \displaystyle\sum_{j=1}^m c^{*}_i\mf K^{\sigma_i,\sigma^{\prime}_j}(\alpha(\alpha(b_i))^*\alpha(\alpha^{-1}(b^{\prime}_j))c^{\prime}_j 
\\&=&\left< \displaystyle\sum_{i=1}^nb_i\otimes e_{\sigma_i} \otimes c_i+K , \displaystyle\sum_{j=1}^m\alpha^{-1}(b^{\prime}_j)\otimes e_{\sigma^{\prime}_j} \otimes c^{\prime}_j+K\right>,
\end{eqnarray*}
we obtain $U^* \left(\displaystyle\sum_{j=1}^m b_i\otimes e_{\sigma^{\prime}_j} \otimes c^{\prime}_j
+K\right)=\displaystyle\sum_{j=1}^m\alpha^{-1}(b^{\prime}_j)\otimes e_{\sigma^{\prime}_j} \otimes c^{\prime}_j+K$. 

Define a sesquilinear form
$[\cdot,\cdot]:\m F\times\m F\to\m C$ as follows:
\[ [f,f']:=\langle f,Uf'\rangle~\mbox{where}~f,f'\in \m F.
 \]
Indeed, for $\displaystyle\sum_{i=1}^nb_i\otimes e_{\sigma_i} \otimes c_i+K$, $\displaystyle\sum_{j=1}^m b^{\prime}_j\otimes e_{\sigma^{\prime}_j} \otimes c^{\prime}_j+K\in \m F$ we obtain
\begin{eqnarray*}
&& \left[\displaystyle\sum_{i=1}^nb_i\otimes e_{\sigma_i} \otimes c_i+K ,\displaystyle\sum_{j=1}^m b^{\prime}_j\otimes e_{\sigma^{\prime}_j} \otimes c^{\prime}_j+K\right] 
 \\&=&  \left<\displaystyle\sum_{i=1}^nb_i\otimes e_{\sigma_i} \otimes c_i+K,\displaystyle\sum_{j=1}^m \alpha(b^{\prime}_j)\otimes e_{\sigma^{\prime}_j} \otimes c^{\prime}_j+K\right> .
\end{eqnarray*}
Thus the module $\left( \m F,\m C,U\right) $ is an S-module. Define the map $
\pi:\m{B}\rightarrow \mathcal{B}^a(\m F)$ by
\begin{equation}
\pi(b)\left(\displaystyle\sum_{i=1}^nb_i\otimes e_{\sigma_i} \otimes c_i+K\right) =\displaystyle\sum_{i=1}^nbb_i\otimes e_{\sigma_i} \otimes c_i+K~  \label{m12}
\end{equation}
$\mbox{for all}~b,b_i\in \m{B};~c_i\in \m C;~\sigma_i\in\Omega$ for $1\leq i\leq n$. We have
\begin{eqnarray*}
 &&\left\|\pi(b)\left(\displaystyle\sum_{i=1}^nb_i\otimes e_{\sigma_i} \otimes c_i+K\right)\right\|^2 
 = \left\|\displaystyle\sum_{i=1}^nbb_i\otimes e_{\sigma_i} \otimes c_i+K\right\|^2
 \\ &=& \left\|\left< \displaystyle\sum_{i=1}^nbb_i\otimes e_{\sigma_i} \otimes c_i+K, \displaystyle\sum_{j=1}^nbb_j\otimes e_{\sigma_j} \otimes c_j+K\right>\right\|
 \\&=& \left\|\displaystyle\sum_{i=1}^n \displaystyle\sum_{j=1}^n c^{*}_i\mf K^{\sigma_i,\sigma_j}(\alpha(bb_i)^*bb_j)c_j\right\|
 \\ &\leq & M(b)\left\|\sum\limits_{i,j=1}^n   c^*_i \mf{K}^{\sigma_i,\sigma_j}(\alpha(b^*_i)b_j) c_j\right\|
  = M(b)\left\|\displaystyle\sum_{i=1}^nb_i\otimes e_{\sigma_i} \otimes c_i+K\right\|^2 
\end{eqnarray*}
where $b,b_i\in \m{B};~c_i\in \m C;~\sigma_i\in\Omega$ for $1\leq i\leq n$.
Thus for each $b\in\m B$, $\pi(b)$ is a well-defined bounded linear operator from $\m F$ to $\m F$. Using
\begin{eqnarray*}
&&\left< \pi(b)\left(\displaystyle\sum_{i=1}^nb_i\otimes e_{\sigma_i} \otimes c_i+K \right),\displaystyle\sum_{j=1}^m b^{\prime}_j\otimes e_{\sigma^{\prime}_j} \otimes c^{\prime}_j+K\right>
 \\&=&  \left< \displaystyle\sum_{i=1}^n bb_i\otimes e_{\sigma_i} \otimes c_i+K ,\displaystyle\sum_{j=1}^m b^{\prime}_j\otimes e_{\sigma^{\prime}_j} \otimes c^{\prime}_j+K\right>
=  \displaystyle\sum_{i=1}^n \displaystyle\sum_{j=1}^m c^{*}_i\mf K^{\sigma_i,\sigma^{\prime}_j}(\alpha(bb_i)^*b^{\prime}_j)c^{\prime}_j
\\&=&  \displaystyle\sum_{i=1}^n \displaystyle\sum_{j=1}^m c^{*}_i\mf K^{\sigma_i,\sigma^{\prime}_j}(\alpha(b_i)^*\alpha(b^*)b^{\prime}_j)c^{\prime}_j
\\ &=& \left< \displaystyle\sum_{i=1}^nb_i\otimes e_{\sigma_i} \otimes c_i+K ,\displaystyle\sum_{j=1}^m \alpha(b^*)b^{\prime}_j\otimes e_{\sigma^{\prime}_j} \otimes c^{\prime}_j+K\right>
\end{eqnarray*}
and
\begin{eqnarray*}
U\pi(b^*)U^*\left(\displaystyle\sum_{i=1}^nb_i\otimes e_{\sigma_i} \otimes c_i+K \right) &=&
U\pi(b^*)\left(\displaystyle\sum_{i=1}^n \alpha^{-1}(b_i)\otimes e_{\sigma_i} \otimes c_i+K\right)
 \\ &=&  U\left(\displaystyle\sum_{i=1}^n b^*\alpha^{-1}(b_i)\otimes e_{\sigma_i} \otimes c_i+K\right)
\\ &=&  \displaystyle\sum_{i=1}^n \alpha(b^*\alpha^{-1}(b_i))\otimes e_{\sigma_i} \otimes c_i+K
\\ &=& \displaystyle\sum_{i=1}^n \alpha(b^*)b_i\otimes e_{\sigma_i} \otimes c_i+K
\end{eqnarray*}
for all $b,b_i\in \m{B};~c_i\in \m C$ and $\sigma_i\in\Omega$ whenever $1\leq i\leq n$, it follows that 
$\pi:\m{B}\rightarrow \mathcal{B}^a(\m F)$
is a well-defined map. Thus $\pi:\m{B}\to \mathcal{B}^a(\m F)$ is an $U$-representation. Consider
$\mf{i}(\sigma)=\lim_{\mu} 1_{\m B}\otimes e_{\sigma} \otimes u_{\mu}+K$ where $(u_{\mu})$ is the approximate identity of $\m C$. 
Thus 
\begin{equation}\label{eqn1}
 \overline{\mbox{span}}\{b\mf{i}(\sigma) c:b\in\m
B,~\sigma\in\Omega,~c\in\m C\}=\m F.
\end{equation}

Finally
\begin{align*}\langle \mf{i}(\sigma),\pi(b)\mf{i}(\sigma')\rangle 
 &=\left< \lim_{\mu} 1_{\m B}\otimes e_{\sigma} \otimes u_{\mu}+K,b\left(\lim_{\mu'} 1_{\m B}\otimes e_{\sigma} \otimes u_{\mu'}+K\right)\right>
 \\& =\lim_{\mu}\lim_{\mu'}\langle  1_{\m B}\otimes e_{\sigma} \otimes u_{\mu}+K, b\otimes e_{\sigma} \otimes u_{\mu'}+K\rangle
 \\& =\lim_{\mu}\lim_{\mu'} u_{\mu}^*\mf K^{\sigma,\sigma^{\prime}}(b)u_{\mu'}=\mf K^{\sigma,\sigma^{\prime}}(b)
 ~\mbox{for every}~b\in\m B~\mbox{and}~\sigma,\sigma'\in\Omega.\qedhere
\end{align*}
\end{proof}

  
We refer the triple $(\m F,U,\mf i)$ of the above theorem as the {\it Kolmogorov decomposition} for $\mf K$ and the property
of the triple described by Equation \ref{eqn1} as 
the {\it minimality} property. If $(\m F',U',\mf i')$ is another minimal Kolmogorov decomposition for $\mf K$ with $U'(b\mf i'(\sigma)c):=\alpha(b)\mf i'(\sigma)c$ for all $b\in \m B;~\sigma\in
\Omega~\mbox{and}~c\in\m C$, then it is easy to
see that $\mf i(\sigma)\mapsto \mf i'(\sigma)$ for each
$\sigma\in\Omega$ is an isomorphism between these decompositions. Thus Kolmogorov decomposition is unique. In the previous theorem if $\m C$ is a von Neumann algebra acting on a Hilbert space $\m H$, 
then we obtain a von Neumann $\m B$-module $\m F'$ by taking the strong operator topology closure of
$\m F$ in $\ma B(\m H,\m F\bigotimes\m  H)$. Define a map $U': \m F'\to \m F'$ by
\[
 U'(f):=\mbox{sot-}\displaystyle\lim_{\alpha} U(f_{\alpha})~\mbox{where
 $f$=sot-$\displaystyle\lim_{\alpha} f_{\alpha}\in \m F'$ with $f_{\alpha}\in \m F'$.}
\]
It is easy to check that $U'$ is a unitary. 
Let $\displaystyle\lim_{\alpha} f_{\alpha}\in \m F'$ where $f_{\alpha}\in \m F$. It is also immediate that 
for all $b\in \m B$, the limit $\mbox{sot-}\displaystyle\lim_{\alpha} \pi(b)f_{\alpha}$ exists.
In the following manner we can extend the $U$-representation
$\pi':\m B\to \ma B^a(\m F)$ to a representation, which we also denote by $\pi'$, of $\m B$ on
$\m F'$:
\[
 \pi'(b)(f):=\mbox{sot-}\displaystyle\lim_{\alpha} \pi(b)f_{\alpha}~\mbox{where $b\in \m B$,
 $f$=sot-$\displaystyle\lim_{\alpha} f_{\alpha}\in \m F'$ with $f_{\alpha}\in \m F$.}
\]
Fix $b\in \m B$. For every $f$=sot-$\displaystyle\lim_{\alpha} f_{\alpha}$ and
$e$=sot-$\displaystyle\lim_{\beta} e_{\beta}\in \m F'$ with
$f_{\alpha},e_{\beta}\in \m F$ we obtain that
\begin{align*}
 \langle \pi'(b^*) f,e\rangle &=\mbox{sot-}\displaystyle\lim_{\beta}\langle \pi'(b^*) f,e_{\beta}\rangle
 =\mbox{sot-}\displaystyle\lim_{\beta}(\mbox{sot-}\displaystyle\lim_{\alpha}\langle e_{\beta},\pi(b^*) f_{\alpha}\rangle)^*
 \\ &=\mbox{sot-}\displaystyle\lim_{\beta}(\mbox{sot-}\displaystyle\lim_{\alpha}\langle e_{\beta},U^* \pi(b)^*Uf_{\alpha}\rangle)^*
 = \langle  f,U^{\prime*}\pi'(b)^*U'e\rangle,
\end{align*}
therefore $\pi'$ is a $U'$-representation, and moreover $(\m F',\m C, U')$ is an S-module. Thus in this case we obtain the Kolmogorov decomposition 
$(\m F',U',\mf i)$ of the $\alpha$-CPD-kernel $\mf K$,
and now onwards we denote it by $(\m F,U,\mf i)$. If we assume $\m B$ to be a von Neumann algebra acting on $\m H$, 
then this S-module is in fact a von Neumann $\m C$-module where, in the minimality condition, the closure is taken under the strong operator topology.  

\section{Bures distance between $\alpha$-CPD-kernels}\label{sec2}
 
Assume $\mf K_1$ and $\mf K_2$ to be elements of $\m K^{\alpha}_{\Omega}(\m B,\m C)$ for some set $\Omega$ and unital $C^*$-algebras $\m B$ and $\m C$.
The Kolmogorov decompositions of $\mf K_m$ obtained using Theorem \ref{thm1} be $(\widehat{\m F_m},U_m,\widehat{\mf i_m})$ and suppose $\pi_m$ is the associated left actions of $\m B$  where $m=1,2$. 
Consider $\m F=\widehat{\m F_1}\oplus \widehat{\m F_2}$, $\pi=\pi_1\oplus \pi_2$, $\mf i_1=\widehat{\mf i_1}\oplus 0$, $\mf i_2=0\oplus \widehat{\mf i_2}$ and $U=U_1\oplus U_2$.
Then $(\m F,\mf i_m,U)$ satisfies  
\begin{align}\label{eqn2}  \mf{K}^{\sigma,\sigma'}_m(b)= \langle \mf{i}_m(\sigma), b
\mf{i}_m(\sigma')\rangle=\langle \alpha(b^*)\mf{i}_m(\sigma), 
\mf{i}_m(\sigma')\rangle~\mbox{for}~\sigma, \sigma'\in \Omega;~b\in \m B~\mbox{and}~m=1,2.
 \end{align} 
A Hilbert $\m C$-module $\m F$ is called a 
 {\it common S-correspondence} for 
 $\mf K_1$ and $\mf K_2$ if there exists a unitary $U$ on $\m F$, an $U$-representation $\pi:\m B\to \ma B^a(\m F)$ and maps $\mf i_m:\Omega\to\m F$ such that Equation \ref{eqn2} is satisfied.
\begin{definition}
 Let $\m B$ and $\m C$ be unital $C^*$-algebras, $\Omega$ be a set and $\mf K_1,\mf K_2\in\m K^{\alpha}_{\Omega}(\m B,\m C)$. 
 Let a Hilbert $\m C$-module $\m F$ be a common S-correspondence for 
 $\mf K_1$ and $\mf K_2$. For every $m=1,2$ let $\ma S(\m F,\mf K_m)$ denote the set of all functions $\mf i_m:\Omega\to \m F$ such that 
$$  \mf{K}^{\sigma,\sigma'}_m(b)= \langle \mf{i}_m(\sigma), b
\mf{i}_m(\sigma')\rangle=\langle \alpha(b^*)\mf{i}_m(\sigma), 
\mf{i}_m(\sigma')\rangle~\mbox{where}~\sigma, \sigma'\in \Omega~\mbox{and}~b\in \m B.
 $$
 Define $$\beta_{\m F}(\mf K_1,\mf K_2):=\mbox{inf}~\{\|\mf i_1(\sigma)-\mf i_2(\sigma)\|:
 \mf i_m\in\ma S(\m F,\mf K_m)~\mbox{for}~m=1,2~\mbox{and}~\sigma\in\Omega\},$$
 and the {\rm Bures distance} between $\mf K_1$ and $\mf K_2$ by
 $$\beta(\mf K_1,\mf K_2):=\mbox{inf}_{\m F}~\beta_{\m F}(\mf K_1,\mf K_2)$$
 where the infimum is over all common S-correspondence $\m F$ for $\mf K_1$ and $\mf K_2$.
\end{definition}
 When $\m C$ is a von Neumann algebra, the Bures distance is determined by the same formula except that the infimum is now taken over a smaller set as seen below:
 
 \begin{lemma}
   Assume $\m B$ to be a unital $C^*$-algebra and $\m C$ to be a von Neumann algebra acting on a Hilbert space $\m H$. 
   Let $\Omega$ be a set and $\mf K_1,\mf K_2\in\m K^{\alpha}_{\Omega}(\m B,\m C)$. 
 Then
 $$\beta(\mf K_1,\mf K_2)=\mbox{inf}_{F}~\beta_{F}(\mf K_1,\mf K_2)$$
 where the infimum is over all common S-correspondences $F$ for $\mf K_1$ and $\mf K_2$ such that $ F$ is also a
 von Neumann $\m C$-module.
 \end{lemma}
 \begin{proof}
 Since each von Neumann $\m C$-module is a Hilbert $\m C$-module,  the inequality 
 $\beta(\mf K_1,\mf K_2)\leq\mbox{inf}_{ F}~\beta_{ F}(\mf K_1,\mf K_2)$ holds, where the infimum is over all common S-correspondences $F$ 
 for $\mf K_1$ and $\mf K_2$ such that $ F$ is also a von Neumann $\m C$-module. Assume that
  a Hilbert $\m C$-module $\m F$ is a common S-correspondence for $\mf K_1$ and $\mf K_2$, then we obtain a
 von Neumann $\m C$-module $F=\overline{span}^s \m F ,$ i.e., 
   the strong operator topology closure of
$\m F$ in $\ma B(\m H,\m F\bigotimes\m  H)$. Since $\m F$ is a subset of $F$, it is also a common S-correspondence for $\mf K_1$ and $\mf K_2$. 
Hence 
$\mbox{inf}_{F}~\beta_{ F}(\mf K_1,\mf K_2)\leq \beta(\mf K_1,\mf K_2)$.
 \end{proof}
\begin{proposition}\label{prop1}
Let $\m B$ be a unital $C^*$-algebra and $\m C$ be a von Neumann algebra acting on a Hilbert space $\m H$. Suppose $\Omega$ is a set. Then there exists
a von Neumann $\m C$-module $\m F$ such that the following holds:
\begin{itemize}
 \item [(i)] $\beta(\mf K_1,\mf K_2)=\beta_{\m F}(\mf K_1,\mf K_2)$ when $\mf K_1,\mf K_2\in\m K^{\alpha}_{\Omega}(\m B,\m C)$;
  \item [(ii)] For each $\mf K_1\in\m K^{\alpha}_{\Omega}(\m B,\m C)$ we obtain an element $\mf i_1\in \ma S(\m F,\mf K_1)$ such that for every $\mf K_2\in\m K^{\alpha}_{\Omega}(\m B,\m C)$ we have
  $$\beta(\mf K_1,\mf K_2)=\mbox{inf}~\{\|\mf i_1(\sigma)-\mf i_2(\sigma)\|:\mf i_2\in \ma S(\m F,\mf K_2),~\sigma\in\Omega\}.$$
  \end{itemize}
\end{proposition}
\begin{proof}
 Let $(\m F_{\mf K},U_{\mf K},\mf i_{\mf K})$ be the Kolmogorov decomposition for $\mf K\in\m K^{\alpha}_{\Omega}(\m B,\m C)$. Let 
 $\m H'=\bigoplus_{\mf K} \m H_{\mf K}$ where 
 $\m H_{\mf K}$ is the interior tensor product of $\m F_{\mf K}$ and $\m H$. Define a von Neumann $\m C$-module $\m F_0$ to be
 the strong operator topology closure of 
 $\bigoplus_{\mf K} \m F_{\mf K}$ in $\ma B(\m H,\m H')$. For each $\mf K$, since $\m F_{\mf K}$ is a subset of $\m F_0$, the set $\ma S(\m F_0, \mf K)$ is nonempty. Define the von Neumann $\m C$-module 
 $\m F$ to be $\m F_0\bigoplus \m F_0$.

$(i)$: Given a common S-correspondence $\m F'$ of $\mf K_1,\mf K_2\in\m K^{\alpha}_{\Omega}(\m B,\m C),$ which also is a von Neumann $\m C$-module, 
we show that 
$\beta_{\m F}(\mf K_1,\mf K_2)\leq\beta_{\m F'}(\mf K_1,\mf K_2)$.
 In fact, this follows if we prove that for each $\mf i_m\in \ma S(\m F',\mf K_m)$ there exist $\widehat{\mf j_m}\in \ma S(\m F,\mf K_m)$ 
 for $m=1,2$ that satisfy $\|\widehat{\mf j_1}(\sigma)-\widehat{\mf j_2}(\sigma)\|\leq \|\mf i_1(\sigma)-\mf i_2(\sigma)\|$ for all $\sigma\in\Omega$. 
 Let $\mf i^{\prime\prime}_1\in \ma S(\m F_0,\mf K_1)$.
 Define a 
bilinear unitary $V:{\overline{span}}^s \m B\mf i^{\prime\prime}_1(\Omega)\m C\to\overline{span}^s \m B\mf i_1(\Omega)\m C$ by
$V(b\mf i^{\prime\prime}_1(\sigma)c):=b\mf i_1(\sigma)c$ for $b\in\m B,~c\in\m C$ and $\sigma\in\Omega$. Assume $P$ to be the bilinear projection
of $\m F'$ onto $\overline{span}^s \m B\mf i_1(\Omega)\m C$. Denote $P(\mf i_2(\sigma))\in \overline{span}^s \m B\mf i_1(\Omega)\m C\subset \m F'$ 
by $\mf j_2(\sigma)$ and denote $(1-P)(\mf i_2(\sigma))\in (\overline{span}^s \m B\mf i_1(\Omega)\m C)^\perp\subset \m F'$ 
by $\mf j^{\prime}_2(\sigma)$ for all $\sigma\in \Omega$. For each $\sigma,\sigma'\in\Omega$ define 
$\mf L^{\sigma,\sigma'}(b):=\langle \mf j_2(\sigma), b\mf j_2(\sigma')\rangle$ and define 
$\mf M^{\sigma,\sigma'}(b):=\langle \mf j^{\prime}_2(\sigma), b\mf j^{\prime}_2(\sigma')\rangle$. It follows that 
$\mf K^{\sigma,\sigma'}_2=\mf L^{\sigma,\sigma'}+\mf M^{\sigma,\sigma'}$  for all $\sigma,\sigma'\in\Omega$. Set 
$\widehat{\mf i_2}(\sigma)=V^*(\mf j_2(\sigma))$ for $\sigma\in\Omega$ which is an element of 
${\overline{span}}^s \m B\mf i^{\prime\prime}_1(\Omega)\m C\subset\m F_0$. Therefore for each $\sigma,\sigma'\in\Omega$ and $b\in\m B$ we have
\begin{align*}\langle\widehat{\mf i_2}(\sigma), b\widehat{\mf i_2}(\sigma')\rangle
&=\langle V^*(\mf j_2(\sigma)), bV^*(\mf j_2(\sigma'))\rangle=\langle V^*(\mf j_2(\sigma)), V^*(b\mf j_2(\sigma'))\rangle
\\&=\langle\mf j_2(\sigma), b\mf j_2(\sigma')\rangle=\mf L^{\sigma,\sigma'}(b). 
\end{align*}
Select any element ${\widetilde{\mf i_2}}\in\ma S(\m F_0,\mf M),$ and define maps $\widehat{\mf j_1}$ and $\widehat{\mf j_2}$ by $\widehat{\mf j_1}(\sigma):=\mf i''_1(\sigma)\oplus 0$ and $\widehat{\mf j_2}(\sigma):=\widehat{\mf i_2}(\sigma)\oplus {\widetilde{\mf i_2}}(\sigma)$ for each $\sigma\in\Omega$. Thus $\widehat{\mf j_m}\in\ma S(\m F,\mf K_m)$ for $m=1,2$ such that
\begin{align*}
\|\widehat{\mf j_1}(\sigma)-\widehat{\mf j_2}(\sigma)\|^2
&=\|\langle\widehat{\mf j_1}(\sigma),\widehat{\mf j_1}(\sigma)\rangle+\langle\widehat{\mf j_2}(\sigma),\widehat{\mf j_2}(\sigma)\rangle-2\mbox{Re}(\langle\widehat{\mf j_1}(\sigma),\widehat{\mf j_2}(\sigma)\rangle)\|
\\&=\|\langle\mf i^{\prime\prime}_1(\sigma),\mf i^{\prime\prime}_1(\sigma)\rangle+\langle\widehat{\mf i_2}(\sigma),\widehat{\mf i_2}(\sigma)\rangle
+\langle{\widetilde{\mf i_2}}(\sigma),{\widetilde{\mf i_2}}(\sigma)\rangle-2\mbox{Re}(\langle\mf i^{\prime\prime}_1(\sigma),\widehat{\mf i_2}(\sigma)\rangle)\|
\\&=\|\langle\mf i^{\prime\prime}_1(\sigma)-\widehat{\mf i_2}(\sigma),\mf i^{\prime\prime}_1(\sigma)-\widehat{\mf i_2}(\sigma)\rangle
+\langle{\widetilde{\mf i_2}}(\sigma),{\widetilde{\mf i_2}}(\sigma)\rangle\|
\\&=\|\langle V(\mf i^{\prime\prime}_1(\sigma)-\widehat{\mf i_2}(\sigma)),V(\mf i^{\prime\prime}_1(\sigma)-\widehat{\mf i_2}(\sigma))\rangle
+\langle{\widetilde{\mf i_2}}(\sigma),{\widetilde{\mf i_2}}(\sigma)\rangle\|
\\&=\|\langle \mf i_1(\sigma)-\mf j_2(\sigma),\mf i_1(\sigma)-\mf j_2(\sigma)\rangle
+\langle\mf j^{\prime}_2(\sigma),\mf j^{\prime}_2(\sigma)\rangle\|
\\&=\|\langle \mf i_1(\sigma), \mf i_1(\sigma)\rangle +\langle\mf j_2(\sigma),\mf j_2(\sigma)\rangle-2\mbox{Re}(\langle\mf i_1(\sigma),\mf j_2(\sigma)\rangle)
+\langle\mf j^{\prime}_2(\sigma),\mf j^{\prime}_2(\sigma)\rangle\|
\\&=\|\langle \mf i_1(\sigma), \mf i_1(\sigma)\rangle +\langle\mf i_2(\sigma),P\mf i_2(\sigma)\rangle-2\mbox{Re}(\langle\mf i_1(\sigma),\mf j_2(\sigma)\rangle)
\\&~~~~~\hspace{0.2in}+\langle\mf i_2(\sigma),(1-P)(\mf i_2(\sigma))\rangle\|
\\&=\|\langle \mf i_1(\sigma), \mf i_1(\sigma)\rangle +\langle\mf i_2(\sigma),\mf i_2(\sigma)\rangle-2\mbox{Re}(\langle\mf i_1(\sigma),\mf j_2(\sigma)\rangle)\|
\\&=\|\langle \mf i_1(\sigma), \mf i_1(\sigma)\rangle +\langle\mf i_2(\sigma),\mf i_2(\sigma)\rangle-2\mbox{Re}(\langle\mf i_1(\sigma),\mf i_2(\sigma)\rangle)\|
\\& ~~~~~~~~~~~~~~~~~~\hspace{0.2in}(\mbox{because}~\mf i_2(\sigma)=\mf j_2(\sigma)\oplus\mf j^{\prime}_2(\sigma)~\mbox{and}~\mf i_1(\sigma)\perp \mf j^{\prime}_2(\sigma))
\\& =\|\langle \mf i_1(\sigma)-\mf i_2(\sigma),  \mf i_1(\sigma)-\mf i_2(\sigma)\rangle\|
\\&=\|\mf i_1(\sigma)-\mf i_2(\sigma)\|^2.
\end{align*}
This is true for all $\mf i_m\in \ma S(\m F',\mf K_m)$ where $m=1,2$; and hence $\beta_{\m F}(\mf K_1,\mf K_2)\leq\beta_{\m F'}(\mf K_1,\mf K_2)$.

$(ii)$: Observe that in part $(i)$ of this proof the choice of $\widehat{\mf j_1}\in\ma S(\m F,\mf K_1)$ does not depend on $\mf K_2$ 
and $\m F'$. On the other hand the choice of $\widehat{\mf j_2}$ depends on $\mf i_1$ and $\mf i_2$ and hence we denote $\widehat{\mf j_2}$ by $\widehat{\mf j_2}(\mf i_1,\mf i_2)$. Then
\begin{align*}
\beta_{\m F}(\mf K_1,\mf K_2)
&=\mbox{inf}~\{\|\mf i(\sigma)-\mf j(\sigma)\|:\mf i\in\ma S(\m F,\mf K_1),~\mf j\in\ma S(\m F,\mf K_2),~\sigma\in\Omega\}
\\&\leq\mbox{inf}~\{\|\widehat{\mf j_1}(\sigma)-\mf j(\sigma)\|:\mf j\in\ma S(\m F,\mf K_2),~\sigma\in\Omega\}
\\&\leq\mbox{inf}~\{\|\widehat{\mf j_1}(\sigma)-\widehat{\mf j_2}(\mf i_1,\mf i_2)(\sigma)\|:\mf i_m\in\ma S(\m F',\mf K_m),~\sigma\in\Omega,~m=1,2\}
\\&=\mbox{inf}~\{\|\mf i_1(\sigma)-\mf i_2(\sigma)\|:\mf i_m\in\ma S(\m F',\mf K_m),~\sigma\in\Omega,~m=1,2\}
\\&=\beta_{\m F'}(\mf K_1,\mf K_2).
\end{align*}
Since this holds for arbitrary common S-correspondence $\m F'$, we have 
\begin{align*}\beta(\mf K_1,\mf K_2)
&\leq\beta_{\m F}(\mf K_1,\mf K_2) \leq\mbox{inf}~\{\|\widehat{\mf j_1}(\sigma)-\mf j(\sigma)\|:\mf j\in\ma S(\m F,\mf K_2),~\sigma\in\Omega\}
\\ &\leq \beta(\mf K_1,\mf K_2).\qedhere
\end{align*}
\end{proof}

\begin{theorem}
Let $\m B$ be a unital $C^*$-algebra and $\m C$ be a von Neumann algebra acting on a Hilbert space $\m H$. Then the function $\beta$ is a metric on $\m K^{\alpha}_{\Omega}(\m B,\m C)$. 
\end{theorem}
 \begin{proof}
For each $\mf K_1,\mf K_2\in\m K^{\alpha}_{\Omega}(\m B,\m C)$, $\beta(\mf K_1,\mf K_2)\geq 0$ and
 $\beta(\mf K_1,\mf K_2)=\beta(\mf K_2,\mf K_1)$. Let $\m F$ and $\widehat{\mf j_1}\in\ma S(\m F,\mf K_1)$ be as in
the proof of Proposition \ref{prop1}$(ii)$. Thus if $\beta(\mf K_1,\mf K_2)= 0,$ 
then $\mbox{inf}~\{\|\widehat{\mf j_1}(\sigma)-\mf j(\sigma)\|:\mf j\in\ma S(\m F,\mf K_2),~\sigma\in\Omega\}=0.$ This yields 
$\widehat{\mf j_1}\in\ma S(\m F,\mf K_2),$ because
$\ma S(\m F,\mf K_2)$ is a norm closed subset of $\m F$. Thus $\mf K_1=\mf K_2$. Moreover, if $\mf K_3\in\m K^{\alpha}_{\Omega}(\m B,\m C)$, then
\begin{align*}\beta(\mf K_2,\mf K_3)
&=\mbox{inf}~\{\|\mf i_2(\sigma)-\mf i_3(\sigma)\|:\mf i_m\in\ma S(\m F,\mf K_m);~m=2,3;~\sigma\in\Omega\}
\\&\leq\mbox{inf}~\{\|\mf i_2(\sigma)-\mf i_1(\sigma)\|:\mf i_2\in\ma S(\m F,\mf K_2);~\sigma\in\Omega\}
\\&~~+\mbox{inf}~\{\|\mf i_1(\sigma)-\mf i_3(\sigma)\|:\mf i_3\in\ma S(\m F,\mf K_3);~\sigma\in\Omega\}
\\&=\beta(\mf K_2,\mf K_1)+\beta(\mf K_1,\mf K_3).
\end{align*}
\end{proof}

 \section{Intertwiners and Rigidity Theorem}

 Assume $\m F$ to be a Hilbert $C^*$-module over a von Neumann algebra $\m C$ acting on a Hilbert space $\m H$.  
 Suppose $\m B$ is a unital $C^*$-algebra and $\Omega$ is a set. Let 
 $\m F$ be a common S-correspondence for $\mf K_1,\mf K_2\in\m K^{\alpha}_{\Omega}(\m B,\m C)$, and $\mf i_m\in \ma S(\m F,\mf K_m)$ for $m=1,2$. Observe that
 \begin{align*}\|\mf i_1(\sigma)-\mf i_2(\sigma)\|^2&=\|\langle \mf i_1(\sigma)-\mf i_2(\sigma),\mf i_1(\sigma)-\mf i_2(\sigma)\rangle\|
 \\&=\|\mf K^{\sigma,\sigma}_1(1_{\m B})+\mf K^{\sigma,\sigma}_2(1_{\m B})-2\mbox{Re}(\langle\mf i_1(\sigma),\mf i_2(\sigma)\rangle)\|~\mbox{for all}~\sigma\in \Omega.
 \end{align*}
 Therefore the set $\{\langle\mf i_1(\sigma),\mf i_2(\sigma)\rangle:\mf i_m\in\ma S(\m F,\mf K_m)~\mbox{for}~m=1,2~\mbox{and}~\sigma\in\Omega\}$
 determines the Bures distance $\beta(\mf K_1,\mf K_2)$. We denote this set by $N_{\m F}(\mf K_1,\mf K_2)$. 
 Indeed, \begin{align}\label{eqn3}\nonumber&\beta_{\m F}(\mf K_1,\mf K_2)\\\nonumber=&\mbox{inf}~\{\|\mf i_1(\sigma)-\mf i_2(\sigma)\|:
 \mf i_m\in\ma S(\m F,\mf K_m)~\mbox{for}~m=1,2~\mbox{and}~\sigma\in\Omega\}
 \\=&\mbox{inf}~\{\|\mf K^{\sigma,\sigma}_1(1_{\m B})+\mf K^{\sigma,\sigma}_2(1_{\m B})-2\mbox{Re}(\langle\mf i_1(\sigma),\mf i_2(\sigma)\rangle)\|^{\frac{1}{2}}:
\langle\mf i_1(\sigma),\mf i_2(\sigma)\rangle\in N_{\m F}(\mf K_1,\mf K_2)\}. \end{align} 
Set $N(\mf K_1,\mf K_2):=\displaystyle\cup_{\m F} N_{\m F}(\mf K_1,\mf K_2)$ where $\m F$ is a common S-correspondence for 
 $\mf K_1$ and $\mf K_2$. Let $\ma B^{a,bil}(\m F_1,\m F_2)$ denote the set of all adjointable bilinear maps between S-correspondences
 $\m F_1$ and $\m F_2,$ and
 $$M(\mf K_1,\mf K_2):=\{\langle \mf i_1(\sigma),T\mf i_2(\sigma)\rangle:T\in \ma B^{a,bil}(\m F_2,\m F_1),~\|T\|\leq 1\}$$
 where $(\m F_m,U_m,\mf i_m)$ is a Kolmogorov decomposition for $\mf K_m$ whenever $m=1,2.$ Elements of $M(\mf K_1,\mf K_2)$ are intertwiners between two (minimal) Kolmogrorov decompositions of $\alpha$-CPD-kernels.
 \begin{proposition}\label{prop2}
Suppose $(\m F_m,U_m,\mf i_m)$ is a Kolmogorov decomposition for $\mf K_m,$ where $m=1,2.$ Then the following statements hold:
\begin{itemize}
 \item [(i)] The definition of $M(\mf K_1,\mf K_2)$ does not depend upon the choice of 
 the Kolmogorov decomposition $(\m F_m,U_m,\mf i_m)$ for $\mf K_m,$ where $m=1,2.$ 
 \item [(ii)] $M(\mf K_1,\mf K_2)=N(\mf K_1,\mf K_2)=N_{\m F_1\oplus \m F_2}(\mf K_1,\mf K_2)$.
 \item [(iii)] $\beta(\mf K_1,\mf K_2)=\mbox{inf}~\{\|\mf K^{\sigma,\sigma}_1(1_{\m B})+\mf K^{\sigma,\sigma}_2(1_{\m B})
 -2\mbox{Re}(m)\|^{\frac{1}{2}}:
m\in M(\mf K_1,\mf K_2)\}.$
\end{itemize}
 \end{proposition}
 \begin{proof}
Assume $(\m F^{\prime}_m,U^{\prime}_m,\mf i^{\prime}_m)$ to be the minimal Kolmogorov decomposition for $\mf K_m$, 
  for each $m=1,2.$ Define 
  $$M^{\prime}(\mf K_1,\mf K_2):=\{\langle \mf i^{\prime}_1(\sigma),T^{\prime}\mf i^{\prime}_2(\sigma)\rangle:
  T^{\prime}\in \ma B^{a,bil}(\m F^{\prime}_2,\m F^{\prime}_1),~\|T^{\prime}\|\leq 1\}.$$
  For each $m=1,2,$ we define a bilinear unitary $V_m$ from $\m F^{\prime}_m$ to $\overline{span}^s \m B\mf i_m(\Omega)\m C$ by 
  $V_m:b\mf i^{\prime}_m(\sigma)c\mapsto b\mf i_m(\sigma)c$ for all $b\in\m B,$ $c\in\m C$ and $\sigma\in \Omega$. 
  Each $V_m\in\ma B^{a,bil}(\m F^{\prime}_m,\m F_m)$ and satisfy $V^*_mV_m=id_{\m F^{\prime}_m}$, because the range 
  $\overline{span}^s \m B\mf i_m(\Omega)\m C$ of $V_m$ is a 
  complemented submodule of $\m F_m$ (cf. \cite[Theorem 3.6]{La95}). This implies that $V_m(\mf i^{\prime}_m(\sigma))=\mf i_m(\sigma)$
and $V^*_m(\mf i_m(\sigma))=\mf i^{\prime}_m(\sigma)$ for each $\sigma\in\Omega$ and $m=1,2.$ If 
$\langle \mf i_1(\sigma),T\mf i_2(\sigma)\rangle\in M(\mf K_1,\mf K_2)$ for some 
$T\in \ma B^{a,bil}(\m F_2,\m F_1)$ with $\|T\|\leq 1$, then $T':=V^*_1TV_2\in \ma B^{a,bil}(\m F^{\prime}_2,\m F^{\prime}_1)$ and 
$\|T'\|\leq 1.$ Further we get $\langle \mf i_1(\sigma),T\mf i_2(\sigma)\rangle\in M^{\prime}(\mf K_1,\mf K_2)$, since
$$ \langle \mf i_1(\sigma),T\mf i_2(\sigma)\rangle=\langle V_1\mf i^{\prime}_1(\sigma),TV_2\mf i^{\prime}_2(\sigma)\rangle
=\langle \mf i^{\prime}_1(\sigma),V^*_1TV_2\mf i^{\prime}_2(\sigma)\rangle=\langle \mf i^{\prime}_1(\sigma),T'\mf i^{\prime}_2(\sigma)\rangle.$$
A similar argument also yields the reverse inclusion $ M^{\prime}(\mf K_1,\mf K_2)\subset M(\mf K_1,\mf K_2)$. This proves statement $(i).$ 

$(ii):$ Let $\m F$ be a common S-correspondence for $\mf K_1$ and $\mf K_2$. Assume 
$\langle\mf i_1(\sigma),\mf i_2(\sigma)\rangle\in N_{\m F}(\mf K_1,\mf K_2),$ $\m F_1=\m F_2=\m F$ and $T=id_{\m F}$. Thus by $(i),$ we get
$$\langle\mf i_1(\sigma),\mf i_2(\sigma)\rangle=\langle\mf i_1(\sigma),T\mf i_2(\sigma)\rangle\in M(\mf K_1,\mf K_2).$$
This is true for all choices of $\langle\mf i_1(\sigma),\mf i_2(\sigma)\rangle\in N_{\m F}(\mf K_1,\mf K_2)$ and $\m F$, it follows that
$N(\mf K_1,\mf K_2)\subset M(\mf K_1,\mf K_2).$ On the other hand suppose that 
$\langle\mf i_1(\sigma),T\mf i_2(\sigma)\rangle\in M(\mf K_1,\mf K_2).$ Fix $\mf j_1(\sigma)=\mf i_1(\sigma)\oplus 0$ and
$\mf j_2(\sigma)=T\mf i_2(\sigma)\oplus \sqrt[]{id_{\m F_2}-T^*T}\mf i_2(\sigma)\in \m F_1\oplus\m F_2.$ Thus for each $b\in\m B$, we obtain 
$\langle \mf j_1(\sigma), b\mf j_1(\sigma^{\prime})\rangle=\langle \mf i_1(\sigma), b\mf i_1(\sigma^{\prime})\rangle=\mf K^{\sigma,\sigma^{\prime}}_1(b)$ and 
\begin{align*}
 \langle \mf j_2(\sigma), b\mf j_2(\sigma^{\prime})\rangle
 &=\left< T\mf i_2(\sigma)\oplus \sqrt[]{id_{\m F_2}-T^*T}\mf i_2(\sigma), 
 bT\mf i_2(\sigma^{\prime})\oplus b~\sqrt[]{id_{\m F_2}-T^*T}\mf i_2(\sigma^{\prime})\right>
  \\&=\left<  T\mf i_2(\sigma),Tb\mf i_2(\sigma^{\prime})\right>
  +\left<  \sqrt[]{id_{\m F_2}-T^*T}\mf i_2(\sigma), \sqrt[]{id_{\m F_2}-T^*T}b\mf i_2(\sigma^{\prime})\right>
  \\&=\left<  \mf i_2(\sigma),T^*Tb\mf i_2(\sigma^{\prime})\right>
  +\left<  \mf i_2(\sigma), (id_{\m F_2}-T^*T )b\mf i_2(\sigma^{\prime})\right>
  \\&=\left<  \mf i_2(\sigma),b\mf i_2(\sigma^{\prime})\right>=\mf K^{\sigma,\sigma^{\prime}}_2(b).
\end{align*}
Similarly $\mbox{for each}~b\in\m B$ we can prove that $\mf 
K^{\sigma,\sigma^{\prime}}_1(b)=\langle \alpha(b^*)\mf j_1(\sigma), \mf j_1(\sigma^{\prime})\rangle$ and $\mf K^{\sigma,\sigma^{\prime}}_2(b)=$\\$
 \langle \alpha(b^*)\mf j_2(\sigma), \mf j_2(\sigma^{\prime})\rangle.$
Therefore $(\m F_1\bigoplus\m F_2, \mf j_m)$ is a Kolmogorov construction for $\mf K_m$ for each $m=1,2.$ Observe that 
$\langle\mf i_1(\sigma),T\mf i_2(\sigma)\rangle=\langle\mf j_1(\sigma),\mf j_2(\sigma)\rangle\in N_{\m F_1\oplus \m F_2}(\mf K_1,\mf K_2).$
This proves $M(\mf K_1,\mf K_2)\subset N_{\m F_1\oplus \m F_2}(\mf K_1,\mf K_2)$, and hence 
$N(\mf K_1,\mf K_2)\subset M(\mf K_1,\mf K_2)\subset N_{\m F_1\oplus \m F_2}(\mf K_1,\mf K_2)\subset N(\mf K_1,\mf K_2).$

The statement $(iii)$ follows from Equation \ref{eqn3}.
\end{proof}
\begin{corollary}\label{cor1}
Suppose $\m B$ is a unital $C^*$-algebra and $\m C$ is a von Neumann algebra acting on a Hilbert space $\m H$. 
  Assume $\Omega$ to be a set and $(\m F_m,U_m,\mf i_m)$ to be a Kolmogorov decomposition for $\mf K_m\in\m K^{\alpha}_{\Omega}(\m B,\m C)$ 
  where $m=1,2.$ Then 
  \begin{align*}\beta(\mf K_1,\mf K_2)&=\beta_{\m F_1\bigoplus \m F_2}(\mf K_1,\mf K_2)
  \\&=\mbox{inf}~\{\|\mf i_1(\sigma)\oplus 0-\mf j_1(\sigma)\|:
  \mf j_1\in \ma S(\mbox{$\m F_1\bigoplus \m F_2,\mf K_2)$}\}.
  \end{align*}
\end{corollary}
\begin{proof}
 If $\langle \mf i_1(\sigma),T\mf i_2(\sigma)\rangle\in M(\mf K_1,\mf K_2),$ then for each 
 $m=1,2,$ there exists $\mf j_m\in \ma S(\m F_1\bigoplus \m F_2,\mf K_m)$ 
  such that $\langle \mf i_1(\sigma),T\mf i_2(\sigma)\rangle
 =\langle \mf j_1(\sigma),\mf j_2(\sigma)\rangle$ and $\mf j_1(\sigma)=\mf i_1(\sigma)\oplus 0$ (cf. the proof of Proposition \ref{prop2}(ii)). 
From Proposition \ref{prop2}(iii) it follows that
\begin{align*}
 &~~~~~\beta(\mf K_1,\mf K_2)
 \\&=\mbox{inf}~\{\|\mf K^{\sigma,\sigma}_1(1_{\m B})+\mf K^{\sigma,\sigma}_2(1_{\m B})
 -2\mbox{Re}(\langle\mf i_1(\sigma)\oplus 0,\mf j_2(\sigma)\rangle)\|^{\frac{1}{2}}:
T\in \ma B^{a,bil}(\m F_2,\m F_1),~\|T\|\leq 1\}
\\&\geq \mbox{inf}~\{\|\mf K^{\sigma,\sigma}_1(1_{\m B})+\mf K^{\sigma,\sigma}_2(1_{\m B})
 -2\mbox{Re}(\langle\mf i_1(\sigma)\oplus 0,\mf j^{\prime}_1(\sigma)\rangle)\|^{\frac{1}{2}}:
 \mf j^{\prime}_1\in \ma S(\mbox{$\m F_1\bigoplus \m F_2,\mf K_2)$}\}
\\&= \mbox{inf}~\{\|\mf i_1(\sigma)\oplus 0-\mf j^{\prime}_1(\sigma)\|:
\mf j^{\prime}_1\in \ma S(\mbox{$\m F_1\bigoplus \m F_2,\mf K_2)$}\}
\\&\geq  \beta_{\m F_1\bigoplus \m F_2}(\mf K_1,\mf K_2).\qedhere
\end{align*}
\end{proof}

Next we obtain a technical proposition for CPD-kernels which defined over a set and are from a fixed von Neumann algebra to itself. This result would
be useful, in particular, to prove a rigidity theorem for CPD-kernels.

\begin{proposition}\label{prop3}
 
Let $\m B$ be a von Neumann algebra acting on a Hilbert space $\m H$. Let $\Omega$ be a set and $(\m F,\mf i)$ be the minimal 
Kolmogorov decomposition for the CPD-kernel $\mf K$ over $\Omega$ from $\m B$ to $\m B$. Then the following statements are equivalent:
\begin{itemize}
 \item [(i)] There exists a unit vector in the center $\ma C_{\m B} (\m F):=\{f\in\m F:bf=fb,~\mbox{for all}~b\in\m B\}$.
 \item [(ii)] The module $\m F$ is isomorphic to $\m B\bigoplus \m F'$ where $\m F'$ is a von Neumann $\m B$ submodule of $\m F$.
 \item [(iii)] There exists a  subset $\{b^\sigma:\sigma\in\Omega\}$ of $\m B$ such that the two sided strongly closed ideal generated by
 $\{b^\sigma:\sigma\in\Omega\}$ equals to $\m B$.
 Moreover, there exists a CPD-kernel $\mf L$ over $\Omega$ from $\m B$ to $\m B$ 
 such that
 \begin{equation}\label{eqn4}
  \mf K^{\sigma,\sigma'}(b)=b^{\sigma*}bb^{\sigma'}+\mf L^{\sigma,\sigma'}(b)~\mbox{for all}~b,b^\sigma,b^{\sigma'}\in\m B~\mbox{and}~
  \sigma,\sigma'\in\Omega.
 \end{equation}
 Indeed, $b\mapsto b^{\sigma*}bb^{\sigma'}$ is also a CPD-kernel from $\m B$ to $\m B$
\end{itemize}

\end{proposition}
\begin{proof}
 $(i)\Rightarrow (ii):$ Choose a unit vector $\zeta\in \ma C_{\m B} (\m F)$. Then using the linear map $b\zeta\mapsto b$ we identify the 
 von Neumann $\m B$-module generated by $\zeta$ (i.e., $\m B\zeta$) with $\m B$. $(ii)$ follows from the decomposition 
 $\m F=\m B\zeta\bigoplus {\m B\zeta}^\perp.$ 
 
 $(ii)\Rightarrow (iii):$ To prove this implication, it is enough to consider $\m F=\m B\bigoplus \m F'.$ Thus for each $\sigma\in \Omega$, 
 there exists $b^\sigma\in\m B$ and $\mf j(\sigma)\in\m F'$ such that  $\mf i(\sigma)=b^\sigma\oplus\mf j(\sigma)$. This implies 
 $\mf K^{\sigma,\sigma'}(b)=\langle \mf i(\sigma),b\mf i(\sigma')\rangle=\langle \mf j(\sigma),b\mf j(\sigma')\rangle+b^{\sigma*}bb^{\sigma'}$ 
 where $b,b^\sigma,b^{\sigma'}\in\m B$ and $\sigma,\sigma'\in\Omega$. Define $\mf L^{\sigma,\sigma'}(b)=\langle \mf j(\sigma),b\mf j(\sigma')\rangle$
 for all $\sigma,\sigma'\in\Omega$.  Therefore it is clear that $\mf L$ is the required kernel and the two sided strongly closed ideal generated by $\{b^\sigma:\sigma\in\Omega\}$ 
 equals to $\m B$.
 
  $(iii)\Rightarrow (i):$ Let $(\m F'',\mf j)$ be the Kolmogorov decomposition for the CPD-kernel
  $b\mapsto b^{\sigma*}bb^{\sigma'}.$  Since the difference of $\mf K$ and the kernel 
  $b\mapsto b^{\sigma*}bb^{\sigma'}$ 
  is the CPD-kernel $\mf L$, we get a bilinear contraction $v:\m F\to \m F''$ satisfying
  $v(\mf i(\sigma))= \mf j(\sigma)~\mbox{for all}~\sigma\in\Omega.$
  Indeed, for fixed $x=\displaystyle\sum^n_{m=1} b^{\sigma_m}\mf i(\sigma_m)b^{\sigma'_m}\in\m F$ we obtain 
\begin{align*}
 &\langle x,x\rangle-\langle vx,vx\rangle
 \\=&\left< \sum^n_{m=1} b^{\sigma_m}\mf i(\sigma_m)b^{\sigma'_m},\sum^n_{m=1} b^{\sigma_m}\mf i(\sigma_m)b^{\sigma'_m} \right>- 
 \left< v\left(\sum^n_{m=1} b^{\sigma_m}\mf i(\sigma_m)b^{\sigma'_m}\right),v\left(\sum^n_{m=1} b^{\sigma_m}\mf i(\sigma_m)b^{\sigma'_m}\right) \right>
 \\=& \displaystyle\sum^n_{m,l=1} b^{\sigma^{\prime*}_m}\mf L(\alpha(b^{\sigma_m})^* b^{\sigma_l})b^{\sigma'_l}\geq 0,
\end{align*}
i.e., $\|vx\|\leq \|x\|$ for each $x\in \m F$. So the map $v$ extends to a contraction from the strong operator topology closure of $\m F$ 
 to 
the strong operator topology closure of $\m F''$. Since all von Neumann modules are self-dual (cf. \cite[Theorem 4.16]{Sk00}), this extended map $v$ is also adjointable. 
Therefore we have a 
positive contraction $w:=v^*v$, and  
$b^{\sigma*}bb^{\sigma'}=\langle \mf i(\sigma),wb\mf i(\sigma')\rangle$ for all $b\in\m B$; $\sigma,\sigma'\in\Omega.$ Since $w$ 
commutes with each $b\in\m B$, 
we infer that $\sqrt{w}\in \m B'\subset \ma B^a(\overline{span}^s \m F)$. Thus we have
\begin{align}\label{eqn5}
b^{\sigma*}bb^{\sigma'}=\langle \sqrt{w}\mf i(\sigma),b\sqrt{w}\mf i(\sigma')\rangle~\mbox{for all}~b\in\m B;~\sigma,\sigma'\in\Omega.
 \end{align}
Because $1_{\m B}$ belongs to the two sided strongly closed ideal generated by the subset $\{b^\sigma:\sigma\in\Omega\}$ of $\m B$, from 
Equation \ref{eqn5} it follows that
there exists $f\in \m F$ such that $b=\langle f,bf\rangle$ for all $b\in\m B$. This implies that $f$ is a unit vector. Indeed, 
\begin{align*}
 \langle bf-fb,bf-fb\rangle&=\langle bf,bf\rangle-\langle bf,fb\rangle-\langle fb,bf\rangle+\langle fb,fb\rangle
 \\& =\langle f,b^*bf\rangle-\langle f,b^*f\rangle b-b^*\langle f,bf\rangle+b^*\langle f,f\rangle b=0,
\end{align*}
i.e., $f\in \ma C_{\m B} (\m F).$
\end{proof}

 Suppose $\m B$ is a von Neumann algebra. The following rigidity theorem for CPD-kernels shows that if the Bures distance between a CPD-kernel over $\Omega$ from $\m B$ to $\m B$ and the identity kernel is less than one, then the S-correspondence arising out of the Kolmogorov decomposition of the CPD-kernel contains a copy of $\m B$:

\begin{theorem}
 Suppose $\m B$ is a von Neumann algebra acting on a Hilbert space $\m H$ and $\Omega$ is a set. Assume $(\m F,\mf i)$ to be a 
Kolmogorov decomposition for the CPD-kernel $\mf K$ over $\Omega$ from $\m B$ to $\m B$ and $\beta(\mf K,id_{\m B})<1$ 
where CPD-kernel $id_{\m B}^{\sigma,\sigma'}:=id_{\m B}$ for each $\sigma,\sigma'\in\Omega$. Then
$\m F$ is isomorphic to $\m B\bigoplus \m F'$ where $\m F'$ is a von Neumann $\m B$-$\m B$-module.
\end{theorem}
\begin{proof}
 It is enough to consider the case when $(\m F,\mf i)$ is the minimal
Kolmogorov decomposition for the CPD-kernel $\mf K$. Assume $\beta(\mf K,id_{\m B})<1-\epsilon$ for some $\epsilon>0$. Note that
$(\m B,1_{\m B}))$ is the Kolmogorov decomposition for $id_{\m B}$. Therefore it follows from Corollary \ref{cor1} that
there exists $\mf j'_1(\sigma)=1_{\m B}\oplus 0,\mf j'_2(\sigma)=b^{\sigma}\oplus \mf j_2(\sigma)\in \m B\bigoplus \m F$
such that $\|\mf j'_1(\sigma)-\mf j'_2(\sigma)\|\leq \beta(\mf K,id_{\m B})+\epsilon<1.$ Indeed, we have 
$\mf K^{\sigma,\sigma'}(b)=\langle \mf j'_2(\sigma),b\mf j'_2(\sigma')\rangle=b^{\sigma*}bb^{\sigma'}+
\langle \mf j_2(\sigma),b\mf j_2(\sigma')\rangle$ and $\|1-b^{\sigma}\|\leq \|\mf j'_1(\sigma)-\mf j'_2(\sigma)\|<1$ for each 
$\sigma,\sigma'\in\Omega$
and $b\in\m B.$ Thus, each $b^{\sigma}$ is invertible, and this implies that two sided strongly closed ideal generated by 
$\{b^\sigma:\sigma\in\Omega\}$ is $\m B$. Hence, the theorem follows using the part $(iii)\Rightarrow (ii)$ of Proposition \ref{prop3}. 
\end{proof}

\section{Transition probability}

 Let $\m B$ and $\m C$ be unital $C^*$-algebras, $\alpha:\m B\to\m B $ be an automorphism, $\Omega$ be a set and $\mf K_1,\mf K_2\in\m K^{\alpha}_{\Omega}(\m B,\m C)$. 
 For a Hilbert $\m C$-module $\m F$ which is a common S-correspondence for 
 $\mf K_1$ and $\mf K_2$, recall that for $m=1,2,$ the set $\ma S(\m F,\mf K_m)$ denote the set of all functions $\mf i_m:\Omega\to \m F$ such that 
$$  \mf{K}^{\sigma,\sigma'}_m(b)= \langle \mf{i}_m(\sigma), b
\mf{i}_m(\sigma')\rangle=\langle \alpha(b^*)\mf{i}_m(\sigma), 
\mf{i}_m(\sigma')\rangle~\mbox{where}~\sigma, \sigma'\in \Omega~\mbox{and}~b\in \m B.
 $$
Define the transition probability $P(\mf K_1,\mf K_2)$ of  $\mf K_1$ and $\mf K_2$ by

$$\ma P(\mf K_1,\mf K_2):=sup~\{\|\langle\mf i_1(\sigma),\mf i_2(\sigma)\rangle\|^2:\mf i_m\in\ma S(\m F,\mf K_m)
~\mbox{for all}~m=1,2,~\mbox{and}~\sigma\in\Omega\}.$$
For an automorphism $\alpha:\m B\to\m B$, let $\ma F(\mf K_1, \mf K_2)$ be the set whose elements are kernels $\mf L$ over the set $\Omega$
from $\m B$ to $\m C$ such that  
$$\|\mf L^{\sigma,\sigma'}(\alpha(b^*)b')\|^2\leq\|\mf K^{\sigma,\sigma}_1(\alpha(b^*)b)\|\|\mf K^{\sigma^{\prime},\sigma'}_2(\alpha(b^{\prime*})b')\|
~\mbox{for all}~b,b'\in\m B;~\sigma,\sigma'\in\Omega.$$

If $\mf i_m\in\ma S(\m F,\mf K_m)$ for $m=1,2$, then for each
$\sigma,\sigma'\in\Omega$ 
$$\mf L^{\sigma,\sigma'}(b):=\langle \mf i_1(\sigma),b\mf i_2(\sigma')\rangle~\mbox{for all}~ b\in\m B$$
defines a kernel over $\Omega$ from $\m B$ to $\m C$.
Indeed, $\mbox{for each}~ b\in\m B$ and $\sigma,\sigma'\in\Omega$ the computation
\begin{align*}
 \|\mf L^{\sigma,\sigma'}(\alpha(b^*)b')\|^2&=\|\langle \mf i_1(\sigma),\alpha(b^*)b'\mf i_2(\sigma')\rangle\|^2
 =\|\langle b\mf i_1(\sigma),b'\mf i_2(\sigma')\rangle\|^2 
 \\&\leq \| b\mf i_1(\sigma)\|^2 \|b'\mf i_2(\sigma')\|^2
 =\|\langle b\mf i_1(\sigma),b\mf i_1(\sigma)\rangle\|\|\langle b'\mf i_2(\sigma'),b'\mf i_2(\sigma')\rangle\|
 \\&=\|\mf K^{\sigma,\sigma}_1(\alpha(b^*)b)\|\|\mf K^{\sigma',\sigma'}_2(\alpha(b^{\prime*})b')\|,
\end{align*}
implies that $\mf L\in \ma F(\mf K_1, \mf K_2).$

 A kernel $\mf K\in\m K^{\alpha}_{\Omega}(\m B,\m C)$ is called {\it unital} if 
 $\mf K^{\sigma,\sigma}(1_{\m B})=1_{\m C}$ for all $\sigma\in\Omega$. 
 We obtain an alternate description of transition probability in the following theorem. Similar result for transition probability is known for states over unital 
 $C^*$-algebras \cite[Theorem 1]{Al83}, and 
for states over unital $*$-algebras \cite[Lemma 2.4]{Uh85}:

\begin{theorem}
 Suppose $\m B$ is a unital $C^*$-algebra and $\m C$ is a von Neumann algebra acting on a Hilbert space $\m H$. Let $\Omega$ be a set. 
 Then the transition probability between
 unital kernels $\mf K_1,\mf K_2\in\m K^{\alpha}_{\Omega}(\m B,\m C)$ satisfies
 \[
  \ma P(\mf K_1,\mf K_2)=sup~\{\|\mf L^{\sigma,\sigma}(1_{\m B})\|^2:\mf L\in \ma F(\mf K_1, \mf K_2)\}.
 \]

\end{theorem}
\begin{proof}
 Assume $(\m F,U)$ to be a common S-correspondence for 
 $\mf K_1$ and $\mf K_2$. Let $\mf i_m\in\ma S(\m F,\mf K_m)$ for all $m=1,2.$ Define $\m F^{\prime}_1$ and $\m F^{\prime}_2$ to be the submodules 
 $\overline{span}\{b\mf i_1(\sigma):b\in\m B,~\sigma\in\Omega\}$ and $\overline{span}\{b\mf i_2(\sigma):b\in\m B,~\sigma\in\Omega\}$ of $\m F$, respectively.
 Suppose $\m F_1$ and $\m F_2$ are von Neumann submodules of $\m F$ defined as the strong operator topology closure of $\m F^{\prime}_1$ and $\m F^{\prime}_2$,
 respectively. Since all von Neumann modules are self-dual, they are complemented 
in all Hilbert $C^*$-modules which contain them as submodules (cf. \cite{Sk06}). Indeed, we get two orthogonal projections $P_1:\m F\to \m F_1$ and 
$P_2:\m F\to \m F_2$ which commutes with each $b\in\m B$. 

For each $\mf L\in \ma F(\mf K_1, \mf K_2),$   we have
\begin{align}\label{eqn6}\|\mf L^{\sigma,\sigma'}(\alpha(b^*)b')\|^2
\leq\|\mf K^{\sigma,\sigma}_1(\alpha(b^*)b)\|\|\mf K^{\sigma',\sigma'}_2(\alpha(b^{\prime*})b')\|
 =\| b\mf i_1(\sigma)\|^2 \|b'\mf i_2(\sigma')\|^2
\end{align}
for all $b,b'\in\m B;~\sigma\in\Omega.$ This implies that the function $(b\mf i_1(\sigma),b'\mf i_2(\sigma'))\mapsto \mf L^{\sigma,\sigma'}(\alpha(b^*)b')$ 
is a densely defined, bounded sesquilinear $\m C$-valued form defined on $\m F_1\times \m F_2$. Since all
von Neumann $\m C$-modules are self-dual, Riesz representation theorem for $\m C$-functionals holds. This yields that there exists a (unique) bounded $\m C$-linear operator $T_1:\m F_1\to\m F_2$
satisfying
\[ \mf L^{\sigma,\sigma'}(\alpha(b^*)b')=\langle T_1b\mf i_1(\sigma),b'\mf i_2(\sigma')\rangle~\mbox{for all}~ b,b'\in\m B~\mbox{and}~\sigma,\sigma'\in\Omega. 
\]
From Inequality \ref{eqn6}, it is clear that $T_1$ is a contraction. Define $T_2:=P_2T_1P_1\in \ma B^a(\m F).$ 
Since $\|T_1\|\leq 1,$ we have $\|T_2\|\leq 1.$ For $b,b',b''\in\m B;~\sigma,\sigma'\in\Omega$ we have
\begin{align*}
 \langle T_2(b(b'\mf i_1(\sigma))),b''\mf i_2(\sigma')\rangle
 &=\langle T_1(bb')\mf i_1(\sigma)),b''\mf i_2(\sigma')\rangle=\mf L^{\sigma,\sigma'}(\alpha(b^{\prime*}b^*)b'')
 \\& =\mf L^{\sigma,\sigma'}(\alpha(b^{\prime*})\alpha(b^*)b'')=\langle T_1 b'\mf i_1(\sigma),\alpha(b^*)b''\mf i_2(\sigma')\rangle
 \\&=\langle bT_2 b'\mf i_1(\sigma)),b''\mf i_2(\sigma')\rangle,
\end{align*}
which implies $\mf L^{\sigma,\sigma'}(\alpha(b^*)b')=\langle T_2b\mf i_1(\sigma),b'\mf i_2(\sigma')\rangle
~\mbox{for all}~ b,b'\in\m B;~\sigma\in\Omega.$ It also yields that $T_2$ commutes with each $b\in\m B$. Russo and Dye (cf. \cite[Theorem I.8.4]{Da96})
proved that for any $\epsilon>0,$ there exist real numbers 
$r_1,r_2
,\ldots,r_n\geq 0$ such that $\sum^n_{k=1}r_k =1$, and there exist unitary elements $U_1,U_2,\ldots,U_n$ in the commutant of $\m B$ which
satisfy
\[
 \left\|T_2-\sum^n_{k=1}U_kr_k\right\|<\epsilon.
\]
This implies 
\begin{align*}\left\|\mf L^{\sigma,\sigma}(1_{\m B})-\sum^n_{k=1}r_k\langle U_k\mf i_1(\sigma),\mf i_2(\sigma)\rangle\right\|
&=\left\|\left< T_2\mf i_1(\sigma),\mf i_2(\sigma)\right>-\sum^n_{k=1}r_k\langle U_k\mf i_1(\sigma),\mf i_2(\sigma)\rangle\right\|
\\&=\left\|\left< \left(T_2-\sum^n_{k=1}r_k U_k\right)\mf i_1(\sigma),\mf i_2(\sigma)\right>\right\|<\epsilon,
\end{align*}
and therefore we get $\|\mf L^{\sigma,\sigma}(1_{\m B})\|<\epsilon+\sum^n_{k=1}r_k\|\langle U_k\mf i_1(\sigma),\mf i_2(\sigma)\rangle\|$. Observe that
$U_k\mf i_1\in \ma S(\m F,\mf K_1),$ for $U_k$ belong to the commutant of $\m B$ where $1\leq k\leq n.$ Thus from the previous estimate and the definition
of transition probability, we conclude that $\|\mf L^{\sigma,\sigma}(1_{\m B})\|<\epsilon+\ma P(\mf K_1,\mf K_2)^{\frac{1}{2}}.$ Because $\epsilon$ is arbitrary,
$\|\mf L^{\sigma,\sigma}(1_{\m B})\|\leq\ma P(\mf K_1,\mf K_2)^{\frac{1}{2}}.$ Thus
$  \ma P(\mf K_1,\mf K_2)\geq sup~\{\|\mf L^{\sigma,\sigma}(1_{\m B})\|^2:\mf L\in \ma F(\mf K_1, \mf K_2)\}.$

Conversely, for each $\epsilon>0,$ there exist $\mf i_1\in \ma S(\m F,\mf K_1),$ $\mf i_2\in \ma S(\m F,\mf K_2)$ and $\sigma\in \Omega$ such that 
$\|\langle \mf i_1(\sigma),\mf i_2(\sigma)\rangle\|^2\geq \ma P(\mf K_1,\mf K_2)-\epsilon.$ 
For each $\sigma\in\Omega,$ define a bounded linear map from $\m B$ to $\m C$ by
$$\mf L^{\sigma,\sigma'}(b):=\langle \mf i_1(\sigma),b\mf i_2(\sigma')\rangle~\mbox{for all}~ b\in\m B.$$
Thus $\mf L\in \ma F(\mf K_1, \mf K_2)$ and 
$\|\mf L^{\sigma,\sigma}(1_{\m B})\|^2=\|\langle \mf i_1(\sigma),\mf i_2(\sigma)\rangle\|^2\geq  \ma P(\mf K_1,\mf K_2)-\epsilon.$
Because this is true for each $\epsilon>0,$ 
 \[
  \ma P(\mf K_1,\mf K_2)\leq sup~\{\|\mf L^{\sigma,\sigma}(1_{\m B})\|^2:\mf L\in \ma F(\mf K_1, \mf K_2)\}.
 \]
 So the theorem holds.
\end{proof}
\vspace{0.5cm} 
\noindent{\bf Acknowledgements:} The first author was supported by the Seed Grant from IRCC, IIT Bombay. 
Harsh Trivedi thanks IRCC, IIT Bombay for the research associateship and ISI Bangalore for the visiting scientist fellowship.

\end{document}